\documentclass[11pt,twoside]{article}
\usepackage{latexsym}
\usepackage{amssymb,amsbsy,amsmath,amsfonts,amssymb,amscd,amsthm}
\usepackage{graphicx}
\usepackage{hyperref}
\usepackage{color}
\newcommand{\commentout}[1]{}

\newcommand{\fer}[1]{(\ref{#1})}

\newcommand{\R}{\mathbb{R}}

\newcommand {\sgn} { {\rm sgn} }
\newcommand {\dv}  { {\rm div} }

\newcommand {\cad} { {\cal D} }

\newcommand {\f}   {\frac}

\newcommand{\beq}{\begin{equation}}
\newcommand{\beqa}{\begin{eqnarray}}
\newcommand{\bea} {\begin{array}{ll}}
\newcommand{\beqan}{\begin{eqnarray*}}
\newcommand{\eeq}{\end{equation}}
\newcommand{\eeqa}{\end{eqnarray}}
\newcommand{\eeqan}{\end{eqnarray*}}
\newcommand{\eea} {\end{array}}
\newtheorem{theorem}{Theorem}[section]

\newtheorem{definition}[theorem]{Definition}

\newtheorem{proposition}[theorem]{Proposition}

\newcommand{\cqfd}{{ \hfill
                       {\unskip\kern 6pt\penalty 500
                       \raise -2pt\hbox{\vrule\vbox to 6pt{\hrule width 6pt
                       \vfill\hrule}\vrule} \par}   }}

\title{\Large \bf Scalar conservation laws with rough (stochastic) fluxes}

\author{ Pierre-Louis Lions$^{1}$, Beno\^ \i t Perthame$^{2}$ and Panagiotis E. Souganidis$^{3,4}$}
\date{\today}

\begin{document}
\maketitle
\vspace*{1.0cm}
\pagenumbering{arabic}

\begin{abstract}

\noindent We develop a pathwise theory for scalar conservation laws with quasilinear multiplicative rough path dependence, a special case being stochastic conservation laws with quasilinear stochastic dependence. We introduce the notion of pathwise stochastic entropy solutions, which is closed with the local uniform limits of paths,  and prove that it is well posed, i.e., we establish existence, uniqueness and continuous dependence, in the form of
pathwise $L^1$-contraction, as well as 
some explicit estimates. Our approach is motivated by  the theory of stochastic viscosity solutions, which was introduced and developed 
by two of the authors, to study fully nonlinear first- and second-order stochastic pde with multiplicative noise. This  theory 
relies on special test functions constructed by inverting locally the flow of the stochastic characteristics. For conservation laws this is best implemented at the level of the kinetic formulation which we follow here.

\end{abstract}
\pagestyle{plain} \vspace*{1.0cm}

\\
\noindent {\bf AMS Class. Numbers.} 35L65; 35R60; 60H15; 81S30

\section{Introduction}
\label{sec:intro}

We are putting forward a theory of pathwise weak solutions for scalar conservation laws in $\R^N$ with quasilinear multiplicative rough path dependence. Since a special case is scalar conservation laws with quasilinear multiplicative stochastic dependence,  we call the equations we consider and their weak solutions stochastic scalar conservation laws (SSCL for short) and stochastic entropy solutions respectively.
\smallskip

\noindent Our approach is based  on the concepts and methods introduced by Lions and Souganidis  \cite{LShjscras98, LShjscras98b} and extended by the same authors in \cite{LShjs,LSsemilinear,LShjscras2000b,LShjscrasEDP} for the theory of pathwise stochastic viscosity solution of fully nonlinear first- and second-order stochastic pde including stochastic Hamilton-Jacobi equations. One of the fundamental tools of this theory is the class of test functions constructed by inverting locally, and at the level of test functions, the flow of the characteristics corresponding to the stochastic first-order part of the equation and smooth initial data. Such approach is best implemented for conservation laws using the kinetic formulation which we follow here.

\smallskip

\noindent Let
\beq\label{flux}
{\bf{A}} =(A_1,...,A_N) \in C^2(\R;\R^N)
\eeq
 be the flux and consider a continuous path
\beq\label{path1}
{\bf W }=(W^1,...,W^N) \in C([0,\infty);\R^N),
\eeq
a special case being
\beq\label{path2}
{\bf{W}}=(W^1,...,W^N) \quad \text{ is an $N$-dimensional Brownian motion.}
\eeq


\noindent We are interested in the SSCL

\begin{equation}\label{scl}
\begin{cases}
du + \displaystyle \sum_{i=1}^{N} (A_i(u))_{x_i} \circ dW^i = 0 \quad \text{ in } \quad \R^N\times(0,\infty), \\[2mm]
u=u^0 \quad \text{ on } \quad \R^N\times\{0\}.
\end{cases}
\end{equation}

\smallskip

\noindent Throughout the paper we adopt the notation and terminology of stochastic calculus. In general $du$ denotes some kind of time differential, while in the case of \fer{path2} is the usual stochastic differential. Similarly in the general setting $\circ$ does not have any particular meaning and can be ignored, while in the stochastic setting 
it denotes the Stratonovich differential. The need to use the latter stems from the fact that we are developing a theory which is closed (stable) on paths in the local uniform topology. It is well known, even in the context of stochastic differential equations, that ordinary differential equations with time dependence converging local uniformly to Brownian paths give rise to sde's the Stratonovich differential. That
a pathwise theory is more appropriate to study \eqref{scl} is also justified from the fact that in the actual stochastic case, taking expectations leads, in view of the properties of the Ito calculus, to terms that are not possible to handle by the available estimates. We discuss this issue further in the last section of this paper.

\smallskip

\noindent If we are using a single rough path, i.e., if $W=W^i$ for all $i=1,...,N$, then \fer{scl}
becomes

\begin{equation}\label{scl1}
\begin{cases}
du + \dv {\bf{A}}(u) \circ dW = 0 \  \text{ in } \ \R^N\times(0,\infty), \\[2mm]
u=u^0 \quad \text{ on } \ \R^N\times\{0\}.
\end{cases}
\end{equation}

\smallskip

\noindent In the case of \eqref{scl1} even for a smooth path it is tempting to say that $u(x,t)=v(x, W(t))$ where $v$ solves the time homogenous problem
\begin{equation}\label{scl100}
\begin{cases}
 v_t + \dv {\bf{A}}(v)=0 \  \text{ in } \ \R^N\times(0,\infty), \\[2mm]
u=u^0 \quad \text{ on } \ \R^N\times\{0\}.
\end{cases}
\end{equation}

\smallskip

\noindent Indeed  a formal calculation implies this fact, which, of course, cannot be true since the solutions of  \eqref{scl100} develop shocks, which are not time-reversible, and the $W$ is not a monotone.

\smallskip

\noindent It is worth reminding the reader the connection between stochastic Hamilton-Jacobi and conservation laws for $N=1$. Indeed consider the stochastic Hamilton-Jacobi equation

\begin{equation}\label{shj}
\begin{cases}
dv +  \displaystyle \sum_{i=1}^{N} A_i(v_x) \circ dW^i = 0 \quad \text{ in } \quad \R \times(0,\infty), \\[2mm]
v=v^0 \quad \text{ on } \quad \R^N\times\{0\}.
\end{cases}
\end{equation}
and observe that it is immediate at least formally (actually this can be made rigorous but leave the details to the reader) that
$u=v_x$ solves the SSCL \fer{scl}.

\smallskip

\noindent It may be  possible to use less regular fluxes by increasing the regularity of the paths but will leave this for future investigation. Such an analysis has been performed in detail for stochastic viscosity solutions. It is also possible to consider $x$-dependent fluxes (we plan to pursue this problem in a future work), i.e., initial value problems of the form

\begin{equation}\label{xscl}
\begin{cases}
du + \displaystyle \sum_{i=1}^{N} (A_i(x,u))_{x_i} \circ dW^i = 0 \quad \text{ in } \quad \R^N\times(0,\infty), \\[2mm]
u=u^0 \quad \text{ on } \quad \R^N\times\{0\}.
\end{cases}
\end{equation}

\smallskip

\noindent If, instead of \fer{path1}, we assume that ${\bf W} \in C^1((0,\infty);\R^N)$, then \fer{scl} is a ``classical'' problem with a well known theory, see, for example, the books by Dafermos \cite{DafBook} and Serre \cite{SerreBook}. The solution can develop singularities in the form of shocks (discontinuities). Hence it is necessary to consider entropy solutions which, although not regular, satisfy the $L^1$ -contraction property established by Kruzkov that yields  uniqueness.

\smallskip

\noindent The goal here is to give a sense to these concepts for general rough paths ${\bf W}$. We base the theory on the kinetic formulation which we review in Section \ref{sec:kf}. Then, following the ideas introduced in \cite{LShjscras98, LShjscras98b} for stochastic Hamilton-Jacobi equations, we define the notion of pathwise stochastic entropy solutions and state our main stability and intrinsic uniqueness result (pathwise $L^1$-contraction) in Section \ref{sec:wsol}. The proofs are presented in Sections \ref{sec:reg} and \ref{sec:conc}.
We also introduce without proofs the notion of pathwise dissipative solutions.

\smallskip

\noindent Our interest to study SCCL is twofold. Given the theory of stochastic viscosity solutions and the connection between conservation laws and Hamilton-Jacobi equations, it is very natural from the mathematical point of view to ask whether there is such a theory for the former. In addition, SCCL like \eqref{scl} arise as models in the theory of mean field games developed by Lasry and Lions (see \cite{ll1}, \cite{ll2}, \cite{ll3}). Next we discuss a concrete example.

\smallskip

\noindent Consider, for $i=1,...,L$, the system of stochastic differential equations
$$
dX^i_t = \sigma (X^i_t, \frac{1}{L-1} \sum_{j\not=i} \delta_{X^j_t}) \circ d{\bf W}_t,
$$
where $\delta_y$ denotes the Dirac mass at $y, {\bf{W}}=(W^1,...,W^N)$ is a $N$-dimensional Brownian motion and the $N \times N$ matrix $\sigma$ is Lipschitz continuous on $\R^N \times P(\R^N)$,  where  $P(X)$ is the space of probability measures on the metric space $X$ endowed with the usual $2$-Wasserstein metric.

\smallskip

\noindent It turns out that, as $L\to\infty$,  the law, $\text{Law} (X^1_t,.,.,.,X^L_t)$, of the $(X^1_t,.,.,.,X^L_t)$ converges in the sense of measures to some $\pi_t \in P(P(\R^N))$, which evolves in time according to

$$
\int U(m)d\pi_t(m) = E[U(m_t)] \ \text{ for all } \ U\in C(P(\R^N)),
$$
where $E$ is the expectation with respect to the probability space associated with ${ \bf W}$ and, for $\sigma^*$  the adjoint of $\sigma$,  $m_t$ solves the SSCL

$$
dm= -\sum_{i=1}^{N}(\sigma^* (x,m)m)_{x_i} \circ dW^i .
$$

\smallskip

\noindent Solutions of deterministic non-degenerate conservation laws have  remarkable regularizing effects in Sobolev spaces of low order. It is an interesting question to see if they are still true in the present case. This is certainly possible with somehow different exponents in view of the case of kinetic equations where stochastic averaging lemmas are established \cite{LPSstav}. 

\smallskip

\noindent It is natural to ask whether the approach we are putting forward for \eqref{scl1} also applies
to SSCL with semilinear rough path dependence like
\begin{equation}\label{scl2}
\begin{cases}
du + \displaystyle\sum_{i=1}^{N} (A_i(u))_{x_i} \circ dW^i  = {\bf \Phi}(u) \circ d {\bf {\tilde W}} \quad \text{ in } \quad \R^N\times(0,\infty), \\[2mm]
u=u^0 \quad \text{ on } \quad \R^N\times\{0\},
\end{cases}
\end{equation}
for  ${\bf \Phi}=(\Phi_1,...,\Phi_m) \in C^2(\R;\R^m)$ and  another m-dimensional path ${\bf {\tilde W}}=({\tilde W}^1,...,{\tilde W}^m)$.

\smallskip

\noindent Recently Debussche and Vovelle \cite{dv} (see also Feng and Nualart \cite{fn} and Chen, Ding and Karlsen \cite{cdk}, Debussche, Hofmanova and Vovelle \cite{DHV}, Hofmanova \cite{Hof1, Hof2}) put forward a theory of weak entropy solutions of scalar conservation laws with Ito-type semilinear (but no stochastic quasilinear dependence), which in our setting take the form
\begin{equation}\label{scl22}
\begin{cases}
du + \displaystyle\sum_{i=1}^{N} (A_i(u))_{x_i} dt  = {\bf \Phi}(u) d {\bf {\tilde W}} \quad \text{ in } \quad \R^N\times(0,\infty), \\[2mm]
u=u^0 \quad \text{ on } \quad \R^N\times\{0\},
\end{cases}
\end{equation}
with $\bf Phi$ as above and ${\bf {\tilde W}}$ an $m$-dimensional Brownian motion.

\smallskip

\noindent The main results of \cite{dv,DHV,Hof1, Hof2}, which also rely on the kinetic theory, is a well posed theory of solutions which satisfy a contraction principle in space and probability. In Section \ref{sec:semilinear} we explain the problem, we present a specific example showing that the approach taken in this paper cannot be used in the semilinear setting, and we discuss why the properties of the stochastic integral yield that a pathwise approach is not possible for \eqref{scl22}.

\smallskip

\noindent It remains an interesting open question to study the well posedness of \eqref{scl1} in view of the fact that they approaches
for \eqref{scl} and \eqref{scl2} appear to be incompatible.



\section{Review of the kinetic formulation}
\label{sec:kf}

We review here the basic concepts of the kinetic theory of scalar conservation laws. We are going to show that it allows us to define a global change of variable along the ``kinetic'' characteristics, a very convenient tool for our purpose. Recall that for the conservation laws in the physical space the characteristics are only defined for short times (before crossing) and the method is not so convenient. Such a conclusion was also drawn in \cite{dv} but for a different reason. There the kinetic setting keeps better track of the entropy dissipation (due to the noise).

\smallskip

\noindent Although we use the notation of the Introduction, throughout the discussion in this section 
we assume that
\beq\label{path3}
{\bf W} \in C^1((0,\infty);\R^N),
\eeq
in which case $du$ stands for the usual derivative and $\circ$ is the usual multiplication and, hence, should be ignored.

\smallskip

\noindent The entropy inequality (see \cite{DafBook, SerreBook}), which guarantees the uniqueness of the weak solutions, is
\beq
\left\{\begin{array}{l}
dS(u) + \displaystyle\sum_{i=1}^{N} (A_i^S(u))_{x_i} \circ dW^i   \leq 0  \quad \text{ in } \quad \R^N\times(0,\infty),
\\ \\
S( u)=S(u^0) \quad \text{ on } \quad \R^N\times\{0\},
\end{array} \right.
\label{eq:rsentr}
\eeq
for all $C^2$  convex functions $S$ and entropy fluxes $A^S$ defined by
$$
\left( {\bf A^S}(u)\right)' = {\bf a}(u) S'(u) \quad \text{ for } \quad
{\bf a} = {\bf A}'.$$

\smallskip

\noindent In view of these inequalities, it appears that a pathwise theory is more appropriate to study \eqref{scl}. Indeed it easily follows, when the paths are smooth, that \eqref{scl} satisfies an $L^1(\R^N)$-contraction property. If in the stochastic setting, i.e., when ${\bf W}$ is actually a Brownian motion, we wanted a theory involving expectations, then Ito-calculus
creates terms -- for simplicity here we take $N=1$-- of the form 
$$
E(\int_{0}^{t}S''(u)(f'(u))^2(u_x)^2 dt,
$$  
which cannot be handled due to the lack of appropriate estimates.

\smallskip

\noindent It is by now well established that the simplest way to handle conservation laws is through their kinetic formulation developed through a series of papers -- see Perthame and Tadmor \cite{PTscal}, Lions, Perthame and Tadmor \cite{LPTscal}, Perthame \cite{Peuniq, PeKF}, and Lions, Perthame and Souganidis \cite{LPS}. The basic idea is to write a linear equation on the nonlinear function

\beq
\chi(x,\xi, t) = \chi (u(x,t), \xi)=   \left\{\begin{array}{l}
+1 \quad \text{ if } \quad 0 \leq \xi \leq u(x,t),
\\[2mm]
-1  \quad \text{ if } \quad u(x,t) \leq \xi \leq 0,
\\[2mm]
\; 0 \quad \text{ otherwise}.
\end{array} \right.
\label{eq:chi}
\eeq

\smallskip

\noindent The kinetic formulation states that using the entropy inequalities (\ref{eq:rsentr}) for all convex entropies $S$ is equivalent  to  $\chi$ solving, in the sense of distributions,

\beq \left\{\begin{array}{l}
d \chi + \displaystyle  \sum_{i=1}^N a_i(\xi)\chi_{x_i} \circ dW^i   = m_{\xi} dt   \  \text{ in }  \  \R^N \times \R \times (0,\infty),
\\ \\
\chi =\chi(u^0(\cdot), \cdot)  \  \text{ on }  \  \R^N\times\R\times\{0\},
\end{array} \right.
\label{eq:skf}
\eeq
where
\beq\label{measure}
m \  \text{ is a nonnegative bounded measure in } \  \R^N \times \R \times (0,\infty).
\eeq

\smallskip

\noindent One direction of this equivalence can be seen, at least formally, easily. Indeed since, for all $(x,t)\in \R^N\times(0,\infty)$,
$$
S\big(u(x,t) \big) -S(0)= \int S'(\xi) \chi \big(u(x,t), \xi \big) d\xi,
$$
multiplying \fer{eq:skf} by $S'(\xi)$ and integrating in $\xi$ leads to  \fer{eq:rsentr}. For the converse see \cite{LPTscal,Peuniq,PeKF}.

\smallskip

\noindent We recall next some basic estimates from the kinetic theory which hold for smooth paths and are, actually, independent of the paths. They are the usual $L^p(\R^N \times (0,T))$ and $BV(\R^N \times (0,T))$ bounds (for all $T>0$) for the solutions, as well as the bounds on the kinetic defect measures $m$, which imply that the latter 
are weakly continuous in $\xi$ as measures on $\R^N\times(0,\infty)$.

\smallskip

\noindent We summarize these properties  in the next two propositions which we state without proof.
\begin{proposition} Assume \fer{flux} and \fer{path3}. The entropy solutions to \fer{scl} satisfy, for all $t>0$,
\beq\label{reg1}
\| u (\cdot,t) \|_{L^p(\R^N)} \leq \| u^0  \|_{L^p(\R^N)}  \quad \text{ for all } \quad p \in [1,\infty],
\eeq
and
\beq\label{reg2}
\| Du (\cdot,t) \|_{L^1(\R^N)} \leq \| Du^0  \|_{L^1(\R^N)}.
\eeq
\label{prop:reg}
\end{proposition}
\begin{proposition} Assume \fer{flux} and \fer{path3}. Then entropy solutions to \fer{scl} satisfy, for all $t>0$,
$$
|\xi|\leq \ |u|\leq \|u^0\|_\infty \quad \text{ in } \quad \{(x,\xi,t)\in \R^n\times\R\times(0,\infty):|\chi(x,\xi,t) > 0 \},
$$
\beq\label{kf1}
\int_{0}^\infty  \int_{\R^{N}} \int_{\R} m(x,\xi,t) dx \; d\xi\; dt  \leq \f 1 2 \| u^0 \|^2_{L^2(\R^N)},
\eeq
\beq\label{kf2}
\int_{0}^\infty  \int_{\R^{N}} m(x,\xi,t)  dx \; dt  \leq \| u^0  \|_{L^1(\R^N)}  \quad \text{ for all } \quad  \xi  \in  \R,
\eeq
and, for all smooth test functions $\psi$,
\beq\label{kf3}
\f d {d\xi} \int_{0}^\infty  \int_{\R^{N}} \psi(x,t) m(x,\xi,t)  dx \; dt  \leq \left[ \| D_{x,t} \psi\|_{L^\infty(\R^{N+1})}+ \| \psi(\cdot,0) \|_{L^\infty(\R^{N})} \right] \| u^0  \|_{L^1(\R^N)}.
\eeq
\label{prop:kf}
\end{proposition}

\smallskip

\noindent The next observation is the backbone for our approach for the SSCL. Its origin goes back to
\cite{LSsemilinear, LShjscras2000b, LShjscras98,LShjscras98b}, where similar arguments for stochastic Hamilton-Jacobi equations form the basis of the theory of stochastic viscosity solutions.

\smallskip

\noindent Since the flux in \fer{scl} is independent of $x$, we can use the characteristics associated with \fer{eq:skf} to derive an identity which is equivalent to solving \fer{eq:skf} in the sense of distributions. Indeed
consider
\beq\label{rho0}
\rho^0 \in \cad(\R^N) \quad \text{ such that} \quad \rho^0 \geq 0 \quad \text{ and } \quad \int_{\R^N} \rho^0(x)dx =1,
\eeq
and observe that
\beq
\rho(y, x,\xi,t)= \rho^0\big(y-x + {\bf a}(\xi){\bf W}(t)\big),
\label{eq:rho}
\eeq
where
\beq\label{rho12}
{\bf a}(\xi) {\bf W}(t):=(a_1(\xi)W^1(t), a_2(\xi)W^2(t),...,a_N(\xi)W^N(t)),
\eeq
solves the linear transport equation (recall that in this section we are assuming that ${\bf W}$ is smooth)
$$
d \rho  + \sum_{i=1}^{N} a_i(\xi)\rho_{x_i} \circ dW^i= 0 \quad \text{ in } \quad \R^N \times \R \times (0,\infty),
$$
and, hence,
\beq\label{rho1}
d( \rho(y, x,\xi,t) \chi(x,\xi,t))  + \sum_{i=1}^{N} a_i(\xi)(\rho (y,x,\xi,t)\chi(x,\xi,t))_{x_i} \circ dW^i  = \rho(y, x,\xi,t) m_{\xi}(x,\xi,t)dt.
\eeq

\smallskip

\noindent Integrating \fer{rho1} with respect to $x$ (recall that $\rho^0$ has compact support) yields that, in the sense of distributions in $\R\times(0,\infty)$,

\beq\label{rho11}
\f d {dt} \int_{\R^{N}} \chi(x,\xi, t) \rho(y, x,\xi,t) dx  = \int_{\R^{N}}  \rho(y, x,\xi,t) m_{\xi}(x,\xi,t) dx.
\eeq

\smallskip

\noindent We remark that although the regularity of the path was used to derive \fer{rho11} the actual conclusion does not need it. In particular \fer{rho11} holds for paths which are only continuous.
Notice also that \fer{rho11} is basically equivalent to the kinetic formulation if the measure $m$ satisfies \fer{measure}.

\smallskip

\noindent Finally we point out that \fer{rho11}  makes sense only after integrating with the respect to $\xi$ against a test function. This requires  that
${\bf a}'\in C^1(\R;\R^N)$ as long as we only use that $m$ is a measure. Indeed, integrating against a test function $\Psi$, we find
$$
\begin{array}{rl}
\int_{\R^{N+1}}&  \Psi(\xi) \rho(y, x,\xi,t) m_{\xi}(x,\xi,t) \ dx d\xi=
\\[2mm]
&=- \int_{\R^{N+1}} \Psi'(\xi)  \rho(y, x,\xi,t) \ m(x,\xi,t) \ dx d\xi\\[2mm]
& + \int_{\R^{N+1}} \Psi(\xi) (\sum_{i=1}^{N} \rho_{x_i}(y,x,\xi,t) a_i'(\xi)W^i(t)) \ m(x,\xi,t) \ dx d\xi
\end{array}
$$
and all the terms make sense as continuous functions tested against a measure.

\smallskip

\noindent We present next some (new) estimates and identities which are needed for the proof of the main results of the paper and are derived from \fer{rho11}. In the statement and later in the paper we write $\delta$ for the Dirac mass at the origin.

\smallskip

\noindent We have:
\begin{proposition}\label{prop:new}
Assume  \fer{flux}, \fer{path3} and $ u^0 \in (L^1\cap L^\infty\cap BV)(\R^N)$. Then, for all $t>0$,
\beq
\f d {dt} \int_{\R^{N+1}} |\chi(x,\xi, t)| dx \; d\xi =-2 \int_{\R^{N}} m(x,0,t) dx,
\label{eq:unpr1}
\eeq
and
\beq
\begin{array}{rl}
\int_{\R^{N+1} }  \int_{\R^{2N}} \delta(\xi - u(z,t))\ \rho(y, z,\xi,t)  \rho(y, x,\xi,t) \ m(t,x,\xi) dx dy dz d\xi
\\[2mm]
=\f 1 2 \f{d}{dt}  \int_{\R^{N+1}}[\left( \int_{\R^{N}} \chi(x,\xi, t) \rho(y, x,\xi,t) dx\right)^2 -|\chi(y,\xi, t)| ] dy d\xi.
\end{array}
\label{eq:unpr2}
\eeq
\end{proposition}

\begin{proof}
The first identity is obtained (see  \cite{Peuniq, PeKF}) multiplying \fer{scl} by $\sgn (\xi)$ and using that $\sgn (\xi) \chi(x, \xi,t)= |\chi(x, \xi,t)|$. Notice that taking the value $\xi=0$ in  $m$ is allowed by the Lipschitz regularity in Proposition \ref{prop:kf}.

\smallskip

\noindent To prove \fer{eq:unpr2}, we use the regularization kernel along the characteristics \fer{eq:rho}. Indeed   \fer{rho11} and the fact that $ \chi_\xi(z,\xi, t) = \delta(\xi)-\delta(\xi - u(z,t))$
yield
\beq\label{eq}
\begin{array}{rl}
\f 1 2 \f{d}{dt}  \int_{\R^{N+1}} &\left( \int_{\R^{N}} \chi(x,\xi, t) \rho(y, x,\xi,t) dx\right)^2 dy d\xi
\\[2mm]
& =  \int_{\R^{N+1}} \left[ \int_{\R^{N}} \chi(z,\xi, t) \rho(y, z,\xi,t) dz \;  \int_{\R^{N}}  \rho(y, x,\xi,t) m_{\xi}(x,\xi,t) \ dx \right]dy  d\xi
\\[2mm]
&=-\int_{\R^{N+1} }  \int_{\R^{2N}} [\delta(\xi)-\delta(\xi - u(z,t))] \rho(y, z,\xi,t) \rho(y, x,\xi,t) \  m(x,\xi,t) dz  dx  dy   d\xi
\\[2mm]
&=-\int_{\R^{N}} m(x,0,t) dx \\[2mm]
& + \int_{\R^{N+1} }  \int_{\R^{2N}}  \delta(\xi-u(z,t)) \rho(y, z,\xi,t) \rho(y, x,\xi,t) \  m(x,\xi,t) dz  dx  dy  d\xi.
\end{array}
\eeq

\smallskip

\noindent An important step in the calculation above is that, for all $\xi \in \R$,
$$
\int_{\R^{N} }  \int_{\R^{2N}} \chi(z,\xi, t) [ D_y\rho(y, z,\xi,t)  \rho(y, x,\xi,t)  +\rho(y, z,\xi,t) D_y \rho(y, x,\xi,t)] \  m(t,x,\xi) dz  dx  dy  =0,
$$
which follows from the observation that the integrand is an exact derivative with respect to $y$.

\smallskip

\noindent Using  \fer{eq:unpr1} in \fer{eq} we get  \fer{eq:unpr2}.
\end{proof}

\smallskip

\noindent We conclude recalling the notion of dissipative solutions which were studied by Perthame and Souganidis \cite{ps1} and are equivalent to the entropy solutions. The interest in them is twofold. Firstly the definition resembles and enjoys the same flexibility as the one for viscosity solutions in, of course, the appropriate function space. Secondly in defining them, it is not necessary to talk at all about entropies, shocks, etc..

\smallskip

\noindent We say that $u \in L^\infty((0,T), (L^1 \cap L^\infty)(\R^N))$ is a dissipative solution of \eqref{scl},  provided that \eqref{path3} holds, if, for all $\Psi \in C([0,\infty),C^\infty_c(\R^N))$ and all $\psi \in C^{\infty,+}_c(\R)$, where the subscript $c$ means compactly supported, and in the sense of distributions,
\begin{equation}\label{dis1}
\frac{d}{dt} \int_{\R^N} \int_{\R} \psi(k) (u-k-\Psi)_+ dxdk \leq
\int_{\R^N} \int_{\R} \psi(k) \sgn_+(u-k-\Psi)(-\Psi_t - \sum_{i=1}^{N} (A_i(u))_{x_i} \circ dW^i) dxdk,
\end{equation}
where $(\cdot)_+$ and $\sgn_+$ denote respectively the positive part and its derivative, and

\smallskip

\noindent To provide an equivalent definition which will allow us to go around the difficulties with inequalities mentioned earlier in this section, we need to take a small detour to recall
the classical fact that, under our regularity assumptions on the flux and paths,
for any $\phi \in C^\infty_c(\R^N)$ and any $t_0>0$, there exists $h>0$, which depends on $\phi$, such that the problem

\begin{equation}\label{dis2}
\begin{cases}
d \bar{\Psi} + \displaystyle \sum_{i=1}^{N} (A_i(\bar{\Psi}))_{x_i} \circ dW^i = 0 \ \text{ in } \ \R^N\times(t_0-h,t_0+h), \\[2mm]
\bar{\Psi}=\phi \quad \text{ on } \  \R^N\times\{t_0\},
\end{cases}
\end{equation}
has a smooth solution given by the method of characteristics.

\smallskip

\noindent We leave it up to the reader to check that the definition of the dissipative solution is equivalent to saying that, for   $\phi \in C^\infty_c(\R^N)$,  $\psi \in C^{\infty,+}_c(\R)$  and any $t_0>0$, there exists $h>0$, which depends on $\phi$, such that in the sense of distributions
\begin{equation}\label{dis3}
\frac{d}{dt} \int_{\R^N} \int_{\R} \psi(k) (u-k-\bar{\Psi})_+ dxdk \leq 0 \ \text{ in } \ (t_0-h,t_0+h),
\end{equation}
where $\bar{\Psi}$ and $h>0$ are as in \eqref{dis2}.

\section{Pathwise stochastic entropy  solutions}
\label{sec:wsol}

Neither the notions of entropy and dissipative solutions nor the kinetic formulation can be used to study \fer{scl}, since both involve either inequalities or quantities with sign which do not make sense for equations/expressions with, in principle, are nowhere differentiable functions. We refer to \cite{LSsemilinear, LShjscras2000b, LShjscras98,LShjscras98b} for a general discussion about the difficulties encountered when attempting to use the classical weak solution approaches to study fully nonlinear stochastic pde.

\smallskip

\noindent  Motivated by the theory of stochastic viscosity solutions (\cite{LSsemilinear, LShjscras2000b, LShjscras98,LShjscras98b}) and \fer{rho11} we introduce next the notion of pathwise stochastic entropy solutions for SSCL. The key fact is the observation in the middle  of the previous section.

\smallskip

\noindent  The basic idea of \cite{LShjscras2000b, LShjscras98,LShjscras98b} is to invert locally the characteristics of the stochastic Hamilton-Jacobi equations to eliminate the stochastic part at the level of the test functions. In our context this is done using the $\rho$'s a fact which leads to \fer{rho11}, which does not have any stochastic terms.

\smallskip

\noindent  We have:

\begin{definition}
Assume \fer{flux} and \fer{path1}. Then $u\in (L^1\cap L^\infty)(\R^N \times (0,\infty))$ is a pathwise stochastic entropy solution to \fer{scl}, if there exists a nonnegative bounded measure $m$ on $\R^N\times \R \times (0,\infty)$ such that, for all test functions $\rho$ given by \fer{eq:rho} with $\rho^0$ satisfying \fer{rho0}, we have, in the sense of distributions in $\R \times (0,\infty)$,
\beq
\f d {dt} \int_{\R^{N}} \chi(x,\xi, t) \rho(y, x,\xi,t) dx  = \int_{\R^{N}}  \rho(y, x,\xi,t) m_{\xi}(x,\xi,t) dx.
\label{eq:conv}
\eeq
\label{def:sol}
\end{definition}

\smallskip

\noindent  The main result is:
\begin{theorem} Assume \fer{flux}, \fer{path1} and $ u^0 \in (L^1\cap L^{\infty})(\R^N)$. There exists a  unique patwise stochastic entropy solution $u \in C\big([0,\infty); L^1(\R^N)\big)\cap L^\infty (\R^N \times (0,T)) $, for all $T>0$,  to \fer{scl} and \fer{reg1}, \fer{reg2}, \fer{kf1}, \fer{kf2} and \fer{kf3} hold.
In addition any stochastic entropy solutions  $u_1, u_2\in C\big([0,\infty); L^1(\R^N)\big)$ to \fer{scl} satisfy, for all $t>0$, the ``contraction'' property
\beq\label{cont}
\| u_2(\cdot,t) -u_1(\cdot,t) \|_{L^1(\R^N)} \leq \| u_2^0-u_1^0 \|_{L^1(\R^N)}.
\eeq
Moreover there exists a uniform constant $C>0$ such that, if, for $i=1,2$, $u_i$ is the stochastic entropy solution of \fer{scl} with path ${\bf W_i}$ and $u_i^0\in BV(\R^N)$, then $u_1$ and $u_2$ satisfy, for all $t>0$, the ``contraction' property''
\beq \label{contraction}
\begin{array}{rl}
& \| u_2(\cdot,t) -u_1(\cdot,t) \|_{L^1(\R^N)} \leq \| u_2^0-u_1^0 \|_{L^1(\R^N)} + 
 C [\|{\bf a}\|(|u^0_1|_{BV(\R^N)} + |u^0_2|_{BV(\R^N)})|({\bf W_1}-{\bf W_2})(t)| \\ [2mm]
& +(\sup_{s\in (0,t)}|({\bf W_1 -W_2})(s)| \|{\bf a'}\|
[ \| u^0_1\|^2_{L^2(\R^{N})} + \| u^0_2\|^2_{L^2(\R^{N})}])^{1/2}].
\end{array}
\eeq
\label{th:mainindep1}
\end{theorem}

\smallskip

\noindent  We remark that, looking carefully into the proof of the \fer{contraction} for smooth paths, it is possible to establish, after some approximations, an estimate similar to \fer{contraction}, for non $BV$-data, with a rate that depends on the modulus of continuity in $L^1$ of the initial data. It is also possible to obtain an error estimate for different fluxes. We leave the details for both to the interested reader.


\smallskip

\noindent  We conclude stating the definition of the pathwise dissipative solution. We leave it to the reader to check that it is equivalent to the one of pathwise entropy solution.

\smallskip

\noindent  As in Section~3 we need to consider local time pathwise smooth in $x$ solutions of the SSCL
\eqref{dis2} which can be constructed, for each $\phi \in C^\infty_c(\R^N)$ and $t_0>0$, using the classical method of characteristics ---see  \cite{LShjscras2000b, LShjscras98,LShjscras98b}. Of course, here the $h$ is now ``random''.

\smallskip

\noindent  With this in mind we have:

\begin{definition}
Assume \eqref{flux} and \eqref{path1}. Then $u \in (L^\infty \cap L^\infty)(\R^N \times (0,\infty)))$ is a pathwise stochastic dissipative solution to \eqref{scl}, if \eqref{dis3} holds for any solution $\Psi$ of \eqref{dis2} given by the method of characteristics in $\R^N \times (t_0-h,t_0+h)$
for any $\phi \in C^\infty_c(\R^N)$ and $t_0>0$.
\end{definition}

\section{Estimates for regular paths}
\label{sec:reg}

Following ideas in \cite{LShjscras98, LShjscras98b}, we think of the solution operator of \fer{scl} as the unique extension of the  solution operators of \fer{scl} with regular paths. Therefore we begin the study of \fer{scl}
with smooth  paths. We obtain estimates that allow us to prove that the solutions corresponding to any regularization of the same path converge to the same limit which is a stochastic entropy solution. Then we prove an intrinsic uniqueness for the latter.

\smallskip

\noindent  The key step is a new estimate, which depends only on the sup-norm of ${\bf W}$ and yields
compactness with respect to time.

\smallskip

\noindent  We have:

\begin{theorem} Assume \fer{flux} and, for $i=1,2$, $u^0_i\in (L^1\cap L^\infty \cap BV)(\R^N)$. Consider two smooth paths  ${\bf W_1}$  and ${\bf W_2}$ and the corresponding solutions $u_1$ and $u_2$ to \fer{scl}.  There exists a uniform constant  $C>0$ such that,
 for all $t>0$, \fer{contraction} holds.
\label{th:timeindep}
\end{theorem}
\begin{proof} We combine the uniqueness proof for scalar conservation laws  based on the kinetic formulation of \cite{Peuniq, PeKF} and the regularization method along the characteristics introduced for Hamilton-Jacobi equations in \cite{LSsemilinear, LShjscras2000b, LShjscras98,LShjscras98b}.

\smallskip

\noindent  To this end we fix $\rho^0$ satisfying \fer{rho0}, we consider the kernels $\rho_1$ and $\rho_2$ corresponding to the paths ${\bf W_1}$  and ${\bf W_2}$, and we write  \fer{eq:conv} for $(\chi_1, \rho_1, m_1)$ and $(\chi_2,\rho_2,m_2)$. After  subtracting  the two equations we find
$$\begin{array}{rl}
 \f d {dt} \int_{\R^{N}} &[ \chi_2(x,\xi, t) \rho_2(y, x,\xi,t) -\chi_1(x,\xi, t) \rho_1(y, x,\xi,t) ]dx =
\\[2mm]
& \int_{\R^{N}} [\rho_2(y,x,\xi,t)m_{2,\xi}(x,\xi,t) -\rho_1(y,x,\xi,t) m_{1,\xi}(x,\xi,t)]dx .
\end{array}$$

\smallskip

\noindent  Multiplying the above expression by  $\int_{\R^{N}} [ \chi_2(x,\xi, t) \rho_2(y, x,\xi,t) - \chi_1(x,\xi, t) \rho_1(y, x,\xi,t) ]dx $ and integrating  in $y$ and $\xi$ we obtain
\beq  \begin{array}{rl}
\f 1 2 \f{d}{dt}  \int_{\R^{N+1}} & \left| \int_{\R^{N}} [ \chi_2(x,\xi, t) \rho_2(y, x,\xi,t) -\chi_1(x,\xi, t) \rho_1(y, x,\xi,t) ] dx \right|^2 dy d\xi
\\[2mm]
&=  \int_{\R^{N+1} }  \int_{\R^{2N}} [ \chi_2(z,\xi, t) \rho_2(y, z,\xi,t) - \chi_1(z,\xi, t) \rho_1(y, z,\xi,t) ] \;
\\[2mm]
&\qquad  [ \rho_2(y, x,\xi,t) m_{2,\xi}(x,\xi,t) -  \rho_1(y, x,\xi,t) m_{1,\xi}(x,\xi,t) ] dz  dx  dy  \ d\xi .
\end{array}
\label{eq:est1}
\eeq

\smallskip

\noindent  The next step is to integrate by parts with respect to $\xi$ in \fer{eq:est1}. This gives rise to  several terms depending on whether the $\xi$ derivative hits either $\chi_1(x,\xi,t)$ and $\chi_2(x,\xi,t)$ or $\rho_2(y,z,\xi,t)$, $\rho_1(y,z,\xi,t)$, $\rho_2(y,x,\xi,t)$ and $\rho_2(y,x,\xi,t)$. We denote the expressions involving these two groups of derivatives by ${\bf I}$ and ${\bf II}$ respectively and we analyze them  separately.

\smallskip

\noindent  We begin with the first term where the integration by parts hits the $\chi$-terms. We find
$$\begin{array}{rl}
{\bf I} &=-\int_{\R^{N+1} }  \int_{\R^{2N}} [ \chi_{2,\xi}(z,\xi, t)\ \rho_2(y, z,\xi,t) - \ \chi_{1,\xi}(z,\xi, t) \rho_1(y, z,\xi,t) ] \;
\\[2mm]
&\qquad  [ \rho_2(y, x,\xi,t) \  m_2(x,\xi,t) -  \rho_1(y, x,\xi,t) \ m_1(x,\xi,t) ] dz  dx  dy   d\xi
\\[2mm]
&=\int_{\R^{N+1} }  \int_{\R^{2N}} [\delta (\xi -u_2(z,t)) \rho_2(y, z,\xi,t) - \delta (\xi - u_1(z,t)) \rho_1(y, z,\xi,t) ] \;
\\[2mm]
&\qquad  [ \rho_2(y, x,\xi,t) \  m_2(x,\xi,t) -  \rho_1(y, x,\xi,t) \ m_1(x,\xi,t) ] dz  dx  dy   d\xi ,
\end{array}$$
since the term containing the Dirac masses at $\xi=0$ vanishes because
$$
\int_{\R^{N} }  [ \rho_2(y, z,0,t) - \rho_1(y, z,0,t) ]dz =0.
$$

\smallskip

\noindent  Observe that, since $\rho^0 \geq0$, the two cross-terms containing $\delta (\xi - u_i(z,t)) \; m_j(x,\xi,t)$ with $i\neq j$ are non-positive.

\smallskip

\noindent  Therefore we end up  with
$$\begin{array}{rl}
{\bf I} &\leq  \int_{\R^{N+1} }  \int_{\R^{2N}} [\delta (\xi -u_2(z,t)) \rho_2(y, z,\xi,t)  \rho_2(y, x,\xi,t) \ m_2(x,\xi,t)
\\[2mm]
&\qquad \qquad + \delta (\xi -u_1(z,t)) \rho_1(y, z,\xi,t)  \rho_1(y, x,\xi,t) \  m_1(x,\xi,t)] dz  dx dy  d\xi,
\end{array}$$
and, in view of \fer{eq:unpr2},
\beq\label{I}
\begin{array}{rl}
{\bf I} \leq& \ \ \f 1 2 \f{d}{dt}  \int_{\R^{N+1}}\left[\left( \int_{\R^{N}} \chi_1(x,\xi, t) \rho_1(y, x,\xi,t) dx\right)^2 -|\chi_1(y,\xi, t)| \right] dy   d\xi \\[2mm]
& + \f 1 2 \f{d}{dt}  \int_{\R^{N+1}}\left[\left( \int_{\R^{N}} \chi_2(x,\xi, t) \rho_2(y, x,\xi,t) dx\right)^2 -|\chi_2(y,\xi, t)| \right] dy  d\xi.
\end{array}
\eeq

\smallskip

\noindent  We continue with the other term where the integrations by parts hit the $\rho$-terms. Recalling \fer{rho12} we have
$$\begin{array}{rl}
{\bf II} &=-\int_{\R^{N+1} } \int_{\R^{2N}} [ \chi_2(z,\xi, t)({\bf a}'(\xi){\bf W_2}(t)) \cdot D \rho_2(y, z,\xi,t) - \chi_1(z,\xi, t) ({\bf a}'(\xi){\bf W_1}(t)) \cdot D \rho_1(y, z,\xi,t) ] \;
\\[2mm]
&\qquad \qquad  [ \rho_2(y, x,\xi,t) \  m_2(x,\xi,t) -  \rho_1(y, x,\xi,t) \ m_1(x,\xi,t) ] dz  dx  dy  d\xi
\\[2mm]
&\quad -\int_{\R^{N+1} } \int_{\R^{2N}} [ \chi_2(z,\xi, t) \rho_2(y, z,\xi,t) - \chi_1(z,\xi, t)  \rho_1(y, z,\xi,t) ] \;
\\[2mm]
&\quad  \quad [({\bf a}'(\xi){\bf W_2}(t)) \cdot D\rho_2(y, x,\xi,t) \  m_2(x,\xi,t) - ({\bf a}'(\xi){\bf W_1}(t)) \cdot D \rho_1(y, x,\xi,t) \ m_1(x,\xi,t) ] dz  dx  dy   d\xi .
\end{array}$$

\smallskip

\noindent  To simplify the representation next we use the notation
$$
\bar{\bf W}= \f{{\bf W^2} + {\bf W^1}}{2}, \quad {\bf {\delta W} }= \f{ {\bf W^2} -  {\bf W_1}}{2} , \quad
 {\bf W^1} = \bar{{\bf W}} +  {\bf {\delta W}} \quad \text{ and } \quad {\bf W^2}= \bar{{\bf W}} -  {\bf {\delta  W}},
$$
write $ {\bf II}= {\bf IIa}+ {\bf IIb},$
and analyze each term separately.

\smallskip

\noindent  We begin with
$$\begin{array}{rl}
{\bf IIa} &=
-\int_{\R^{N+1} } {\bf a'}(\xi) {\bf {\delta W}}(t) \cdot \int_{\R^{2N}} [ \chi_2(z,\xi, t) D \rho_2(y, z,\xi,t) +\chi_1(z,\xi, t) D \rho_1(y, z,\xi,t) ]\;
\\[2mm]
&\qquad \qquad  [ \rho_2(y, x,\xi,t) \  m_2(x,\xi,t) -  \rho_1(y, x,\xi,t) \ m_1(x,\xi,t) ] dz  dx  dy   d\xi
\\[2mm]
&\quad -\int_{\R^{N+1} } {\bf a'}(\xi) {\bf {\delta W}}(t) \cdot \int_{\R^{2N}} [ \chi_2(z,\xi, t) \rho_2(y, z,\xi,t) - \chi_1(z,\xi, t)  \rho_1(y, z,\xi,t) ] \;
\\[2mm]
&\qquad  \qquad [D \rho_2(y, x,\xi,t) \  m_2(x,\xi,t) + D \rho_1(y, x,\xi,t) \ m_1(x,\xi,t) ] dz  dx  dy  d\xi ,
\end{array}$$
which can be reorganized as
$$\begin{array}{rl}
{\bf IIa} &= \int_{\R^{3N+1} } {\bf a'}(\xi) {\bf {\delta W}}(t) \cdot
\\[2mm]
& \big[- \chi_2(z,\xi, t) [D \rho_2(y, z,\xi,t)  \rho_2(y, x,\xi,t)+ \rho_2(y, z,\xi,t)  D \rho_2(y, x,\xi,t) ]  m_2(x,\xi,t)
\\[2mm]
&+  \chi_1(z,\xi, t) [D \rho_1(y, z,\xi,t)  \rho_1(y, x,\xi,t)+ \rho_1(y, z,\xi,t)  D \rho_1(y, x,\xi,t) ]  m_1(x,\xi,t)
\\[2mm]
& + \chi_2(z,\xi, t) [D \rho_2(y, z,\xi,t) \rho_1(y, x,\xi,t)-  \rho_2(y, z,\xi,t) D \rho_1(y, x,\xi,t) ]m_1(x,\xi,t)
\\[2mm]
&- \chi_1(z,\xi, t) [D \rho_1(y, z,\xi,t) \rho_2(y, x,\xi,t) \\[2mm]
& - \rho_1(y, z,\xi,t) D \rho_2(y, x,\xi,t) ]m_2(x,\xi,t)  \big]dz  dx dy  d\xi.
\end{array}$$

\smallskip

\noindent  Interpreting the $D$ as a derivative in $y$, for $i=1,2$, we have
$$
\int_{y\in \R^N} [D_y \rho_i(y, z,\xi,t)  \rho_i(y, x,\xi,t)+ \rho_i(y, z,\xi,t) D_y \rho_i(y, x,\xi,t) ]dy =0,
$$

\smallskip

\noindent  and, for the same reason, after integrating by parts in the two remaining lines, we get
$$
\begin{array}{rl}
{\bf IIa} & =
 2 \int_{\R^{3N+1} } {\bf {a'}}(\xi){\bf {\delta W}}(t)\cdot \big[ \chi_2(z,\xi, t) D_y \rho_2(y, z,\xi,t) \rho_1(y, x,\xi,t)  m_1(x,\xi,t)
\\[2mm]
&+ \chi_1(z,\xi, t) D_y \rho_1(y, z,\xi,t) \rho_2(y, x,\xi,t)m_2(x,\xi,t)]\ dx  dy dz d\xi.
\end{array}
$$
\smallskip

\noindent Thus
\beq\label{IIa}
|{\bf IIa}| \leq 2|{\bf {\delta W}}(t)|\;  \|{\bf a'}\|_\infty \; \| D \rho^0 \|_1 \; \| \rho^0 \|_1  \; \int_{\R^{N+1}} \big(m_1(x,\xi,t)+m_2(x,\xi,t)\big) dx  d\xi .
\eeq

\smallskip

\noindent  We turn next to
$$\begin{array}{rl}
{\bf IIb} &=- \int_{\R^{N+1} } {\bf a'}(\xi)\bar {{\bf W}}(t) \cdot\int_{\R^{2N}} [ \chi_2(z,\xi, t) D \rho_2(y, z,\xi,t) - \chi_1(z,\xi, t) D \rho_1(y, z,\xi,t) ] \;
\\[2mm]
&\qquad \qquad  [ \rho_2(y, x,\xi,t) \  m_2(x,\xi,t) -  \rho_1(y, x,\xi,t) \ m_1(x,\xi,t) ] dz \ dx \ dy  \ d\xi
\\[2mm]
&\quad - \int_ {\R^{N+1} } {\bf a'}(\xi)\bar {{\bf W}}(t) \cdot \int_{\R^{2N}} [ \chi_2(z,\xi, t) \rho_2(y, z,\xi,t) - \chi_1(z,\xi, t)  \rho_1(y, z,\xi,t) ] \;
\\[2mm]
&\qquad  \qquad [ D \rho_2(y, x,\xi,t) \  m_2(x,\xi,t) -  D \rho_1(y, x,\xi,t) \ m_1(x,\xi,t) ] dz  dx  dy   d\xi .
\end{array}$$

\smallskip

\noindent  As before the integrals involving  $\chi_2 m_2$ and  $\chi_1 m_1$ vanish because they are exact $y$-derivatives. In addition the terms involving
$\chi_2 m_1$ and $\chi_1 m_2$ cancel after integration by parts in $y$. Hence we conclude that
\beq\label{IIb}
{\bf IIb } =0.
\eeq

\smallskip

\noindent  Combining \fer{eq:est1}, \fer{I}, \fer{IIa} and \fer {IIb} we arrive at the estimate
\beq  \begin{array}{rl}
\f{d}{dt}  \int_{\R^{N+1}} & \left( \int_{\R^{N}} \big[ \chi_2(x,\xi, t) \rho_2(y, x,\xi,t) -\chi_1(x,\xi, t) \rho_1(y, x,\xi,t) \big] dx \right)^2 dy  d\xi \leq
\\[2mm]
& \f{d}{dt}  \int_{\R^{N+1}}\left[\left( \int_{\R^{N}} \chi_1(x,\xi, t) \rho_1(y, x,\xi,t) dx\right)^2 -|\chi_1(y,\xi, t)| \right] dy  d\xi
\\[2mm]
& + \f{d}{dt}  \int_{\R^{N+1}}\left[\left( \int_{\R^{N}} \chi_2(x,\xi, t) \rho_2(y, x,\xi,t) dx\right)^2 -|\chi_2(y,\xi, t)| \right] dy   d\xi
\\[2mm]
& + 4|{\bf \delta W}(t)|\;  \|{\bf a'}\|_\infty \; \|D \rho \|_1 \; \| D \rho^0 \|_1 \;  \int_{\R^{N+1}}, \big(m_1(x,\xi,t)+m_2(x,\xi,t)\big) dx d\xi ,
\end{array}
\label{eq:est2}
\eeq
which is the fundamental inequality that allows us to control the differences between  $u_1$ and $u_2$.

\smallskip

\noindent  It remains to ``eliminate'' the convolution terms in \fer{eq:est2} in order to prove  \eqref{contraction}. To this end,
we use the families of $(\rho^{\eta})_{\eta>0}$ and $(\rho_i^\eta)_{\eta>0}$  given by
$$
\rho^{\eta}(z)= \eta^{-N} \rho^0 ({\eta}^{-1}z) \quad \text{ and } \quad  \rho_i^\eta (y,x,\xi,t)= \rho^{\eta}\left(y-x- {\bf a}(\xi) {\bf W^i}(t)\right),
$$

\smallskip

\noindent  The  $BV$ bounds on $u_1^0$ and $u_2^0$ yield the existence of  a uniform constant $C>0$ such that, for $i=1,2$ and all $t>0$,
$$  \begin{array}{l}
\left| \int_{\R^{N+1}} \left[\left( \int_{\R^{N}} \chi_i(x,\xi, t) \rho^\eta_i(y, x,\xi,t) dx\right)^2 -|\chi_i(y,\xi, t)| \right] dy   d\xi \right|
\\[2mm]
\qquad = \left| \int_{\R^{N+1}}\left[\left( \int_{\R^{N}} \chi_i(x,\xi, t) \rho^{\eta}(y-x) dx\right)^2 -(\chi_i(y,\xi, t))^2 \right] dy   d\xi \right|
\\[2mm] \qquad \leq C |u^0_i|_{BV(\R^N)} \; \eta.
\end{array}
$$

\smallskip

\noindent  For some other uniform constant $C_2>0$, which depends only on
the flux, we also have
$$  \begin{array}{rl}
&\int_{\R^{N+1}}  \left( \int_{\R^{N}} \big[ \chi_2(x,\xi, t) \rho_2^\eta(y, x,\xi,t) -\chi_1(x,\xi, t) \rho_1^\eta (y, x,\xi,t) \big] dx \right)^2 dy  d\xi
\\[2mm]
& =\int_{\R^{N+1}}  \left( \int_{\R^{N}} \big[ \chi_2(x,\xi, t) \rho^\eta(y- x) -\chi_1(x,\xi, t) \rho^\eta(y- x-{\bf a}(\xi){\bf \delta W}(t)) \big] dx \right)^2 dy  d\xi
\\[2mm]
& =\int_{\R^{N+1}} [ \int_{\R^N} ([\chi_2(x,\xi, t)-\chi_2(y,\xi, t)] \rho^\eta(y- x) - [\chi_1(x,\xi, t)-\chi_1(y,\xi, t)]  \rho^\eta(y- x-{\bf a}(\xi){\bf  \delta W}(t)) \\[2mm]
& + [\chi_2(y,\xi,t)\rho^\eta(y-x) - \chi_1(y,\xi, t) \rho^\eta(y- x-{\bf a}(\xi){\bf  \delta W}(t))])dx ]^2 \ dy d\xi
\\[2mm]
&\geq \int_{\R^{N+1}}  \left| \chi_2(y,\xi, t)- \chi_1(y,\xi, t)] \right|^2 dy d\xi
\\[2mm]
& \qquad \qquad -C \int_{\R^{N+1}} \int_{\R^{N}} |\chi_2(x,\xi, t)-\chi_2(y,\xi, t)| \rho^\eta(y- x) dx dy d\xi
\\[2mm]
& \qquad \qquad -C \int_{\R^{N+1}} \int_{\R^{N}} |\chi_1(x,\xi, t)-\chi_1(y,\xi, t)| \rho^\eta(y- x-{\bf a}(\xi) {\bf \delta W} (t)) dx dy  d\xi .
\end{array}
$$

\smallskip

\noindent  A straightforward integration yields
$$
\int_{\R^{N+1}}  \left| \chi_2(y,\xi, t)- \chi_1(y,\xi, t) \right|^2 dy  d\xi = \| (u_1-u_2)(t)\|_{L^1(\R^N)},
$$
while, by a standard convolution estimate, we have
$$
\int_{\R^{N+1}} \int_{\R^{N}} |\chi_2(x,\xi, t)-\chi_2(y,\xi, t)| \rho^\eta(y- x) dx dy  d\xi \leq C\eta \ |u^0_2|_{BV(\R^N)},
$$
and
$$  \begin{array}{rl}
& \int_{\R^{N+1}} ( \int_{\R^{N}} |\chi_1(x,\xi, t)-\chi_1(y,\xi, t)| \rho^\eta(y- x-{\bf a}(\xi) {\bf  \delta W} (t)) dx dy  d\xi
 \\[2mm]
& = \int_{\R^{N+1}}  \int_{\R^{N}} |\chi_1\big(y- \eta z -{\bf a}(\xi) {\bf \delta W}(t),\xi, t\big)-\chi_1(y,\xi, t)| \rho(z) dz dy  d\xi
 \\[2mm]
& \leq C (\eta +\|{\bf a} \|_\infty  |{\bf  \delta W}(t)|)\  |u^0_1|_{BV(\R^N)} .
\end{array}
$$

\smallskip

\noindent  Finally it is immediate that
$$
\begin{array}{rl}
&\int_{\R^{N+1}}  \left( \int_{\R^{N}} \big[ \chi_2(x,\xi, 0) \rho_2^\eta(y, x,\xi,0) -\chi_1(x,\xi, t) \rho_1^\eta (y, x,\xi,0) \big] dx \right)^2 dy  d\xi
\\[2mm]
& = \int_{\R^{N+1}}  \left( \int_{\R^{N}}( \chi_2(x,\xi, 0) -\chi_1(x,\xi,0))\rho ^\eta(y -x )dx \right)^2 dy  d\xi
\\[2mm]
& \leq \|(u_1-u_2)(0)\|_{L^1(\R^N)}
\end{array}
$$

\smallskip

\noindent Integrating \fer{eq:est2} with respect to $t$ and using the previous estimates as well as \fer{kf1} we find,
for some uniform constant $C>0$ which may change from line to line,

\begin{equation}\label{last}
\begin{array}{rl}
& \| (u_1-u_2)(t)\|_{L^1(\R^N)}  \leq  \|(u_1-u_2)(0)\|_{L^1(\R^N)} \\[2mm]
& + C [ (\eta + \|{\bf a}\| |{\bf \delta W} (t)| )  (|u^0_1|_{BV(\R^N)} + |u^0_2|_{BV(\R^N)})\\[2mm]
& + \eta^{-1} \|{\bf a'}\| \int_0^t |{\bf \delta W}(s)|  \int_{\R^{N+1}}, \big(m_1(x,\xi,s)+m_2(x,\xi,s)\big) dx d\xi ]\\[2mm]
& \leq \| (u_1-u_2)(t)\|_{L^1(\R^N)}  \leq  \|(u_1-u_2)(0)\|_{L^1(\R^N)} + C [ (\eta + \|{\bf a}\| |{\bf \delta W} (t)| )  (|u^0_1|_{BV(\R^N)}  + |u^0_2|_{BV(\R^N)}) +  \\[2mm]
& (\sup_{s\in(0,t)}|{\bf \delta W}|) \|{\bf a'}\|
[ \| u^0_1\|^2_{L^2(\R^{N})} + \| u^0_2\|^2_{L^2(\R^{N})}] ].
\end{array}
\end{equation}

\smallskip

\noindent  Choosing  $\eta^2 =\sup_{s\in(0,t)}|{\bf \delta W}| \|{\bf a'}\|
[ \| u^0_1\|^2_{
L^2(\R^{N})} + \| u^0_2\|^2_{L^2(\R^{N})}]$ completes the proof of \fer{contraction}.

\end{proof}

\section{The proof of Theorem \ref{th:mainindep1}}
\label{sec:conc}

In this section we present the

\begin{proof}[The proof of Theorem \ref{th:mainindep1}]
The existence of a stochastic entropy solution follows easily. Indeed, the estimate of Theorem \ref{th:timeindep} implies that, for every $u^0 \in (L^1 \cap L^\infty \cap BV)(\R^N)$ and for every $T>0$, the mapping
$$
{\bf W} \in C([0, T];\R^N) \mapsto u \in C\big([0,T]; L^1(\R^N)\big)
$$
is well defined and uniformly continuous with the respect to the norm of  $C([0,T];\R^N)$. Therefore it has a unique extension to $C([0, T])$ by density. Passing to the limit we recover the ``contraction'' properties \fer{cont} and  \fer{contraction} as well as \fer{def:sol}. Once \fer{cont} is available for initial data in $BV(\R^N)$, the extension to general data is immediate by density.

\smallskip

\noindent Next we show that stochastic entropy solutions satisfying \fer{def:sol} are intrinsically unique in a stronger sense. The contraction property only proves uniqueness of the solution built by the above regularization process but we can prove that property \eqref{eq:conv} ensures uniqueness. Indeed, for $BV$-data the estimates performed in the proof of Theorem \ref{th:timeindep} only use the equality of Definition \ref{def:sol}. From there the only nonlinear manipulation performed is to say that
$$
\begin{array}{rl}
 \f 1 2 \f{d}{dt}&  \int_{\R^{N+1}} \left( \int_{\R^{N}} \chi(x,\xi, t) \rho(y, x,\xi,t) dx\right)^2
\\[2mm]
&=   \int_{\R^{N+1}} \left( \int_{\R^{N}} \chi(x,\xi, t) \rho(y, x,\xi,t) dx\right)
\;  \f{d}{dt}  \int_{\R^{N+1}} \left( \int_{\R^{N}} \chi(x,\xi, t) \rho(y, x,\xi,t) dx\right).
\end{array}
$$

\smallskip

\noindent  This is well justified after time regularization by convolution because we have assumed that solutions belong to  $C\big([0,T); L^1(\R^N)\big)$ for all $T>0$. This fact also allows to justify that the right hand side
$$
  \int_{\R^{N+1}} \left( \int_{\R^{N}} \chi(x,\xi, t) \rho(y, x,\xi,t) dx\right) \ \int_{\R^{N+1}} \int_{\R^{N}} \chi(z,\xi, t) \rho(y, z,\xi,t) dz \;  \int_{\R^{N}}  \rho(y, x,\xi,t) m_{\xi}(x,\xi,t) \ dx
$$
\smallskip

\noindent  \smallskip

\noindent can be analyzed by usual integration by parts because it is possible to add a convolution in $\xi$ before forming the square. All these technicalities are standard and have been detailed in \cite {Peuniq, PeKF}. The uniqueness for general data requires one more layer of approximation.
\end{proof}


\section{The semilinear problem }
\label{sec:semilinear}


It is natural to expect that the approach developed earlier will also be applicable to the semilinear problem \fer{scl2} to yield a pathwise theory of stochastic entropy solutions. It turns out however, as we explain below, that this not the case.

\smallskip

\noindent To keep things simple here we assume that ${\bf W} =t$ and $\bf{ \tilde W}={ \tilde W} \in  C([0,\infty);\R)$ is a single rough path, and consider, for $\Phi \in C^2(\R;\R)$, the problem
\begin{equation}\label{scl3}
\begin{cases}
du + \dv {\bf A}(u)dt = \Phi (u) \circ dW \ \text{ in } \ \R^N\times(0,\infty), \\[2mm]
u=u^0 \ \text{ on } \ \R^N\times\{0\}.
\end{cases}
\end{equation}

\smallskip

\noindent  Following the earlier considerations as well as analogous problem for Hamilton-Jacobi equations (\cite{LSsemilinear} we assume that, for each $v\in\R$ and $T>0$,  the initial value problem
\beq\label{ivp}
\begin{cases}
d\Psi ={\bf \Phi}(\Psi) \circ d{\tilde W} \ \text{ in } \ (0,\infty),\\[2mm]
\Psi(0)=v,
\end{cases}
\eeq
admits a unique solution
\beq\label{ivp1}
\Psi(v;\cdot)\in C([0,T];\R) \quad \text{ such that, for all $t\in [0,T]$,} \quad  \Psi( \cdot,t)\in C^1(\R;\R).
\eeq


\smallskip

\noindent  According to  \cite{LSsemilinear}, to study \fer{scl3} it is natural to consider a change of unknown given by the Doss-Sussman-type transformation
\beq\label{ds}
u(x,t):=\Psi(v(x,t),t).
\eeq

\smallskip

\noindent  Assuming for a moment that $ {\tilde W}$ and, hence, $\Psi$ are smooth with respect to $t$
and \eqref{scl3} and \fer{ivp} have classical solutions we find, after a straightforward calculation, that

\beq\label{scl5}
\begin{cases}
v_t + \dv {\bf {\tilde A}} (v,t) = 0 \ \text{ in } \R^N \times (0,T), \\[2mm]
v=u^0 \quad \text{ on } \quad \R^N\times\{0\},
\end{cases}
\eeq
where $ {\bf {\tilde A}} \in C^{1,0}(\R\times [0,T])$ is given by
\beq\label{newA}
{\bf {\tilde A}'}(v,t) = {\bf A'}(\Psi(v,t)).
\eeq

\smallskip

\noindent  Under the above assumptions on the flux and the forcing term, the theory of entropy solutions of scalar conservation applies to \fer{scl5} and yields the existence of a unique entropy solution.

\smallskip

\noindent  Hence, exactly as in \cite{LSsemilinear}, it is tempting to define $u\in (L^1\cap L^\infty)(\R^N \times (0,T))$, for all $T>0$, to be a pathwise stochastic entropy solution of \eqref{scl3} if $v\in (L^1\cap L^\infty)(\R^N \times (0,T))$ defined, for all $T>0$, by \fer{ds} is an entropy solution of \fer{scl5}.

\smallskip

\noindent  We show next that this is not the case. We begin explaining the difficulty which is best seen
when adding a small viscosity $\nu$ to \eqref{scl3}, and, hence, considering the approximate equation
$$
u_t + \dv {\bf A}(u)= {\bf \Phi}(u) \circ dW + \nu \Delta u ,
$$
and, after the transformation \eqref{ds}, the problem
\begin{align*}
v_t + {\bf a}\big( \Psi(v(x,t),t)\big) .\nabla v &=   \f{\nu}{\Psi_v(v(x,t),t)} \Delta \Psi(v(x,t),t)
\\ \\
&= \nu \Delta v + \nu (\f{\Psi_{vv}}{\Psi_v})(v(x,t),t) |\nabla v|^2 .
\end{align*}

\smallskip

\noindent  If the approach based on \eqref{ds} were correct, one would expect to get, after letting $\nu \to 0$, (rigorously) \eqref{scl5}. This, however, does not seem to be the case due to the lack of the necessary a priori bounds to pass in the limit.

\smallskip

\noindent  The problem is, however, not just a technicality but something deeper. Indeed  the transformation \eqref{ds} does not, in general, preserve the shocks unless, as an easy calculation shows, the forcing is linear.

\smallskip

\noindent  Indeed assume that $N=1$ and  $W(t)=t$ and consider the semilinear Burger's equation
\begin{equation}\label{sclsemi}
\begin {cases}
u_t  + \f  12 (u^2)_x   =  \Phi(u)  \  \text{ in } \ \R \times(0,\infty), \\[2mm]
u^0= 1 \ \text{ if } \ x<0 \ \text{ and} \ 0 \ \text{ if } \ x>0,
\end{cases}
\end{equation}
with $ \Phi$ such that
\begin{equation}\label{phi}
 \Phi(0)=0, \quad  \Phi(1)=0, \ \text{ and } \   \Phi(u) >0  \ \text{ for } \ u \in (0,1).
\end{equation}

\smallskip

\noindent  It is easily seen that the entropy solution of  \eqref{sclsemi} is
$$
u(x,t)= \begin{cases}
            1 \ \text{for } \ x < t/2 , \\
            0 \ \text{for } \ x > t/2 .
        \end{cases}
$$

\smallskip

\noindent  Next we perform the transformation $u = \Psi(v,t)$ with
$$
\dot  \Psi(v;t) = \Phi(\Psi(v;t)), \qquad \Psi(v;0) =v.
$$

\smallskip

\noindent  Since, in view of \eqref{phi},
$$
\Psi(0;t)=\Psi(1;t) \equiv 1, \ \text{ and } \ \Psi(v;t)>v \ \text{ for } \ v \in (0,1),
$$
we find that the flux for the equation for $v$ is
$$
\widetilde A(v,t) =  \int_0^v \Psi(w;t) dw,
$$
and the entropy solution with initial data $u^0$ is $v(x,t)= H(x- \bar x (t))$ with the Rankine-Hugoniot condition
$$
\dot{\bar x}(t) = \int_0^1 \Psi(w;t) dw >  \int_0^1 w dw = \f 12.
$$

\smallskip

\noindent  It is, therefore, clear that the shock waves are not preserved.

\smallskip

\noindent  We conclude with  another reason why it is more natural to consider contractions in $L^1(\R^N \times \Omega)$ instead of 
$L^1(\R^N)$ a.s. in $\omega$ for \eqref{scl2}. To fix the ideas we take ${\bf A}=0$ and we consider the initial value problem
(stochastic de) \begin{equation}\label{scl10}
\begin{cases}
du = \Phi(u) \circ dW \ \text{ in } \ (0,\infty), \\[2mm]
u(0)=u^0 .
\end{cases}
\end{equation}

\smallskip

\noindent  It $u_1$, $u_2$ are solutions to \eqref{scl10} with initial data
$u^0_2, u^0_2$ respectively, then, subtracting the two equations, multiplying by $\text{sign}_{\pm}(u_1-u_2)$, taking expectations and using Ito-calculus, we find, for some $C>0$ depending on bounds on $\Phi$ and its derivatives,
$$    
E \int |u_1(\cdot,t)-u_2(\cdot,t)|\leq \exp(Ct) E \int |u^0_1-u^0_2|,
$$
while it is not possible, in general, to get an almost sure inequality on $\int_|u_1(\cdot,t;\omega)-u_2(\cdot,t,\omega)|$.


\bigskip
\bigskip

\noindent ($^{1}$) Universit\'e Paris-Dauphine, \\
 place du Mar\'echal-de-Lattre-de-Tassigny, \\
 75775 Paris cedex 16, France \\
email: lions@ceremade.dauphine.fr
\\ \\
\noindent ($^{2}$) Universit\'e Pierre et Marie Curie, Paris06,\\
      CNRS UMR 7598 Laboratoire J.-L. Lions, BC187,\\
       4, place Jussieu,  F-75252 Paris 5\\ 
       and INRIA Paris-Rocquencourt, EPC Bang \\
email: benoit.perthame@upmc.fr
\\ \\
($^{3}$)  Department of Mathematics \\
             University of Chicago \\
             Chicago, IL 60637, USA \\
            email: souganidis@math.uchicago.edu
\\ \\
($^{4}$)  Partially supported by the National Science Foundation.

\end{document}